\documentclass[10pt]{amsart}
\usepackage{amsmath,amssymb,amsbsy,amsfonts,amsthm,latexsym,
                        amsopn,amstext,amsxtra,euscript,amscd,mathrsfs,color,bm}
                        
\usepackage{float} 
\usepackage[english]{babel}
\usepackage{mathtools}
\usepackage{url}
\usepackage[colorlinks,linkcolor=blue,anchorcolor=blue,citecolor=blue,backref=page]{hyperref}
\usepackage{enumitem}

\usepackage{setspace}

\newtheorem{theorem}{Theorem}
\newtheorem{lemma}{Lemma}

\newtheorem{remark}{Remark}
\newtheorem{prop}{Proposition}

\newtheorem{observation}{Observation}

\def\cH{{\mathcal H}}

\newtheorem{corollary}{Corollary}

\def\\{\cr}
\def\({\left(}
\def\){\right)}
\def\[{\left[}
\def\]{\right]}
\def\<{\langle}
\def\>{\rangle}

\def\le{\leqslant}
\def\ge{\geqslant}

\def\Hb1{\overline{\cH}_{m}}
\def\Ht1{\widetilde{\cH}_{m}}

\keywords{Apéry sequences, modular forms, rational approximations, Eichler integrals, Hauptmoduln}
\subjclass[2020]{11F11, 11J72, 11M06, 33C20}

\title[]{Modular Forms and Numerical Explorations of Rational Approximations to \texorpdfstring{$\zeta(3)$}{zeta(3)}}

\author{Cynthia Bortolotto  }
\address{IMPA, Brazil}
\email{cynthiab@impa.br}

\author{Lucas Oliveira  }
\address{UFRGS, Brazil}
\email{lucas.oliveira@ufrgs.br}

\begin{document}

\begin{abstract}
We revisit Beukers' modular-form proof of the irrationality of $\zeta(3)$ from the point of view of the auxiliary weight two modular form.  For the Fricke group $\Gamma_0(6)^\star$, we show that Beukers' choice is not isolated: it belongs to a one-parameter affine family.  These approximations have the same exponential decay as the classical Apéry approximations and satisfy the same denominator-growth estimate needed in Beukers' irrationality argument.  We then apply the same construction to several other genus-zero Fricke groups.  
\end{abstract}

\maketitle

\section{Introduction}
The Riemann zeta function $\zeta(s)$ remains one of the most significant objects in analytic number theory, primarily due to its deep-rooted connection with the distribution of prime numbers. In the 18th century, Leonhard Euler established this link through his product formula and succeeded in calculating the values of $\zeta(s)$ at even integers. He demonstrated that $\zeta(2n) = (-1)^{n-1} \frac{(2\pi)^{2n} B_{2n}}{2(2n)!}$, where $B_{2n}$ are the Bernoulli numbers, proving that $\zeta(2n)$ is always a rational multiple of $\pi^{2n}$. Despite three centuries of subsequent research, the values of $\zeta(s)$ at odd integers remain notoriously difficult to characterize. While it is conjectured that these values are transcendental, and that $\pi, \zeta(3), \zeta(5), \ldots$ are algebraically independent over $\mathbb{Q}$, proving even their irrationality has been an open problem for years. 

A significant milestone was reached in 1978 when Roger Apéry proved the irrationality of $\zeta(3)$. His proof relied on Dirichlet's criterion, which states that a real number $\xi$ is irrational if there exist infinitely many coprime integers $p, q$ such that$$\left| \xi - \frac{p}{q} \right| < \frac{c}{q^{r}}$$for constants $c > 0$ and $r > 2$. Apéry constructed a sequence of rational approximations $a_n/b_n$ generated by a second-order recurrence relation, which converged rapidly enough to $\zeta(3)$ to satisfy the criterion. Following Apéry’s breakthrough, the search for irrationality proofs of higher odd zeta values turned toward the construction of linear forms in zeta values with rational coefficients. 

In 2000, Ball and Rivoal \cite{BallRivoal2001} introduced a core method using well-poised hypergeometric series of the form
\begin{align*}
    n!^{s-2r}\sum_{t=0}^\infty \frac{\prod_{j=1}^{(2r+1)n}(t-rn+j) }{ \prod_{j=0}^n (t+j)^{ s+1}},
\end{align*}
which are $\mathbb{Q}$-linear combinations of $1$ and odd zeta values, where $n$ is even and $s$ is odd. Together with Nesterenko's linear independence criterion, they proved that the dimension of the vector space spanned by $\{1, \zeta(3), \zeta(5), \ldots, \zeta(a)\}$ over $\mathbb{Q}$ grows logarithmically with $a$. This implies that infinitely many odd zeta values are irrational. Fischler, Sprang, and Zudilin \cite{Fischler2019} later established a lower bound for the density of $\operatorname{dim}\operatorname{Span}(1, \zeta(3), \zeta(5), \ldots, \zeta(s)).$ Shortly thereafter, Zudilin \cite{Zudilin2001, Zudilin2007} refined this method using well-poised hypergeometric series, allowing him to narrow the focus to specific sets; he proved, for instance, that at least one of the four values in the set $\{\zeta(5), \zeta(7), \zeta(9), \zeta(11)\}$ must be irrational. Although many developments have been made, it remains unknown whether $\zeta(s)$ is irrational for any odd integer $s > 3$.

The ``miraculous" nature of the recurrence relations used by Apéry was later given a geometric foundation by Fritz Beukers \cite{Beukers1979, Beukers1987}. He replaced Apéry's recurrences with explicit iterated integrals and Legendre polynomials. He also noted that if $A(t) = \sum_{n=0}^\infty a_nt^n,$ where $a_n$ are the Apéry numbers, then the recurrence relation they satisfy implies that $\sqrt{A}$ is a solution to the second-order equation
$$ (t^3 - 34t^2 + t)y'' + (2t^2 -51t +1)y' + \frac{1}{4}(t-10)y =0.$$
Changing variables according to the local exponents of this Fuchsian equation, one recovers precisely the Picard-Fuchs equation for the modular family of elliptic curves of $\Gamma_0(6).$ By reframing the problem in the language of modular forms, Beukers provided an alternative proof for the irrationality of $\zeta(2)$ and $\zeta(3)$.

 Yang \cite{Yang2004} explored the differential equations satisfied by modular forms of various weights, while Zagier \cite{Zagier2009} computationally searched for ``Apéry-like" sequences—solutions to certain recurrence equations possessing the required integral properties—and found several sporadic solutions parameterized by modular forms, as in the Beukers case. Yang, in \cite{Yang2008}, described a method to determine Apéry limits of differential equations with solutions given by modular forms. 

In the work that follows, we use this modular lens to explore the construction of rational approximations for $\zeta(3)$. Unlike approaches that focus primarily on the associated Picard-Fuchs equations, our analysis centers on the modular forms themselves and the geometry of the underlying groups $\Gamma_0(N)^\star$. We perform a detailed study of Beukers' method within $\Gamma_0(6)^\star$, demonstrating that the weight 2 modular form $E$---a central component of the construction---is not unique. By characterizing an entire family of such forms, we generate distinct sequences of rational approximations for $\zeta(3)$. Finally, we extend this investigation to other genus-zero groups to explore the limits of this method, specifically identifying the arithmetic and geometric obstructions that prevent these generalizations from yielding valid irrationality proofs.

The main contribution of this paper is twofold. First, we establish a new degree of freedom in the modular construction of Apéry-like sequences. We prove that Beukers’ weight two modular form belongs to a one-parameter affine family; specifically, for any $\alpha \in \mathbb{Q}$, the form $E_b(\tau)(1 + \alpha t_6(\tau))$ generates rational approximations to $\zeta(3)$ sharing the same exponential rate as the classical Apéry approximations. Second, we test the limits of this construction through experiments with other genus-zero Fricke groups. These results demonstrate that the analytic ``flexibility" of the affine family is ultimately constrained by a geometric rigidity: the decay of the linear forms is strictly governed by the first branch value of the Hauptmodul. This comparative analysis provides a clear explanation for why the level $6$ construction is uniquely successful, whereas analogous cases are obstructed by insufficient radii of convergence.

The remainder of this work is organized as follows. In Section 2, we establish the general framework for Beukers' modular approach to the irrationality of $\zeta(3)$. We identify the necessary geometric conditions on a congruence subgroup $\Gamma_0(N)$ and its extension $\Gamma_0(N)^\star$ that allow for the construction of Apéry-like sequences.  Section 3 provides a close re-examination of Beukers' proof in the context of the group $\Gamma_0(6)$. While the original proof relies on specific modular functions satisfying certain conditions, we demonstrate that there exists a larger family of modular functions for which the irrationality proof remains valid.  Finally, in Section 4, we conduct an experimental study of this method across various other congruence groups. By shifting the underlying geometry to groups such as $\Gamma_0(10)^\star, \Gamma_0(14)^\star$, and $\Gamma_0(15)^\star$, we generate new sequences of rational approximations that converge to $\zeta(3)$. We analyze the convergence rates of these sequences, showing that the ``speed" of the approximation is fundamentally determined by the location of the branch points of the Hauptmodul on the boundary of the fundamental domain. This section concludes with an investigation into why certain groups yield slower convergence.

 \section{Outline of the Method}

 The central strategy of this method, originally developed by Beukers, is to construct a sequence of rational approximations to $\zeta(3)$ by exploiting the arithmetic properties of modular forms on certain congruence subgroups. 

 A simplified step-by-step breakdown of the construction is as follows:

 \begin{itemize}
     \item \textbf{Find the Target Modular Form ($F$)}
         We start by searching for a specific modular form, $F(\tau)$, on a congruence subgroup. This form is chosen so that its $L$-function, when evaluated at a specific integer, exactly equals our target value, $\zeta(3)$.

     \item \textbf{Build the Eichler Integral ($f$)}
 Using our modular form $F(\tau)$, we construct a related function called an ``Eichler integral'', denoted by $f(\tau)$. Hecke's Lemma implies that we can connect the Fourier coefficients of $F(\tau)$ to its $L$-function, associated with $\zeta(3)$ in the previous step. This allows us to define a new function, $h(\tau) = f(\tau) - \zeta(3)$, which obeys a specific functional equation and acts as the foundation for our approximations.

     \item \textbf{Constructing a complementary modular form ($E$)}
     Here, we construct a second modular form, $E(\tau)$, which must have rational Fourier coefficients and transform in a way that perfectly complements the symmetries of $h(\tau)$.

     \item \textbf{Combine into a Generating Series}
   We multiply our two constructions together: $E(\tau)(f(\tau) - \zeta(3))$.
 Because their modularity properties complement each other, this product yields a well-behaved power series. The coefficients of this power series take the form $A_n - \zeta(3)B_n$, where $A_n$ and $B_n$ are rational numbers with controlled denominators.

     \item \textbf{Accelerate Convergence with a Hauptmodul ($t$)}
  To make these approximations converge faster, we consider a change in local variables. We replace the standard variable with a new local coordinate, $t(\tau)$, known as a Hauptmodul. By rewriting our power series in terms of $t(\tau)$, we generate a new sequence of rational approximations $a_n/b_n$.

     \item \textbf{Determine the Convergence Rate}
         The rate of convergence of the sequence depends entirely on the radius of convergence of this new series and the denominators of $a_n$ and $b_n$. We calculate the radius of convergence by examining the geometry of the underlying modular group—specifically, the location of the branch points (vertices) on the boundary of the fundamental domain.
     \end{itemize}

To formalize this strategy, we establish the necessary definitions. Let $\Gamma_0(N)$ be the standard congruence subgroup of $SL_2(\mathbb{Z})$ of level $N$, defined by:
$$ \Gamma_0(N) = \left\{ \begin{pmatrix} a & b \\ c & d \end{pmatrix} \in \operatorname{SL}_2(\mathbb{Z}); ~~~~~c \equiv 0 \operatorname{mod} N \right\}.$$

We are particularly interested in the extension of this group by the Fricke involution $W_N$, defined by the matrix $\begin{pmatrix} 0 & -1 \\ N & 0 \end{pmatrix}$. This involution acts on the upper half-plane $\mathbb{H}$ via the map $\tau \mapsto -1/(N\tau)$. We denote the extended group by $\Gamma_0(N)^\star$, which is the group generated by $\Gamma_0(N)$ and $W_N$. Let $M_k(\Gamma_0(N))$ denote the space of modular forms of weight $k$ on $\Gamma_0(N)$. To relate these forms to zeta values, we utilize the Mellin transform of the modular form, which yields its $L$-function. The following lemma, attributed to Hecke, provides the precise functional equation and analytic behavior required to link the Fourier coefficients of a modular form to the values of its $L$-series inside the critical strip.

The following lemma is a key ingredient for the rest of the proof. 

\begin{lemma}[Hecke lemma]
\label{HeckeLemma}
    Suppose that $F\in M_{k}(\Gamma_0(N))$ satisfies the functional equation
    $$
    F\left(-\frac{1}{N\tau}\right)=\epsilon (-i\sqrt{N}\tau)^{k}F(\tau)
    $$
    for some $\epsilon\in\{-1,1\}$, and that $F$ has the following Fourier expansion in a neighborhood of $i\infty$ 
    $$
    F(\tau)=\sum_{n=1}^{\infty}a_{n}q^{n}\,
    $$
    where $q := e^{2\pi i \tau}.$
    Define its Eichler integral as
    $$
    f(\tau)=\sum_{n=1}^{\infty}\frac{a_{n}}{n^{k-1}}q^{n}
    $$
    and its $L$-function $L(F,s)$ as
    $$
    L(F,s)=\sum_{n=1}^{\infty}\frac{a_{n}}{n^{s}}.
    $$
    Then, the function $h:\mathbb{H}\to \mathbb{C}$ given by
    \begin{equation}\label{hecke_lemma1}
        h(\tau)=f(\tau)-\sum_{0\le r<\left[\frac{k-2}{2}\right]}\frac{L(F,k-1-r)}{r!}(2\pi i \tau)^{r}
    \end{equation}
    satisfies the functional equation
    \begin{equation}\label{hecke_lemma2}
        h(\tau)-\mathcal{D}=(-1)^{k-1}\epsilon(-i\sqrt{N}\tau)^{k-2}h\left(-\frac{1}{N\tau}\right)
    \end{equation}
    where the constant $\mathcal{D}$ is given by
    \begin{equation}\label{hecke_lemma3}
        \mathcal{D}=\left\{\begin{array}{cc}
            0\,, & \mbox{ if }k\mbox{ is odd} \\
            \frac{L(F,\frac{k}{2})(2\pi i \tau)^{\frac{k}{2}-1}}{(\frac{k}{2}-1)!}\,, & \mbox{ if }k\mbox{ is even}
        \end{array}\right.
    \end{equation}
    and $L(F,\frac{k}{2})=0$ if $\epsilon=-1$.
\end{lemma}
A proof can be found at \cite{weil}.

The link between the congruence modular group and zeta values is given by the identity
\begin{equation}\label{LfunctionEisenstein}
    L(E_{k}(d\,\cdot\,),s)=-d^{-s}\frac{2k}{B_{k}}\zeta(s)\zeta(s-k+1)\,.
\end{equation}

\begin{lemma}\label{lemma:E2_modular_combinations}
For every integer $d>1$, the function
\begin{equation*}
    \Phi_d(\tau)=E_2(\tau)-dE_2(d\tau)
\end{equation*}
is a modular form of weight $2$ on $\Gamma_0(d)$ and has rational Fourier coefficients. Consequently, whenever $N$ is divisible by $d$, the same function belongs to $M_2(\Gamma_0(N))$.

More generally, if
\begin{equation*}
    E(\tau)=\sum_{d\mid N}\beta_d E_2(d\tau)
\end{equation*}
and
\begin{equation}\label{eq:E2_modularity_condition}
    \sum_{d\mid N}\frac{\beta_d}{d}=0,
\end{equation}
then $E$ is a weight two modular form on $\Gamma_0(N)$ with rational Fourier coefficients whenever the $\beta_d$ are rational.
\end{lemma}

\begin{proof}
The modularity of $\Phi_d$ follows from the usual transformation formula for the quasimodular Eisenstein series $E_2$. If \eqref{eq:E2_modularity_condition} holds, then
\begin{equation*}
    \sum_{d\mid N}\beta_dE_2(d\tau)
    =
    -\sum_{\substack{d\mid N\\ d>1}}\frac{\beta_d}{d}\left(E_2(\tau)-dE_2(d\tau)\right),
\end{equation*}
so the given expression is a linear combination of the modular forms $\Phi_d$. The assertion about rational Fourier coefficients is immediate from the $q$-expansion of $E_2$.
\end{proof}

In what follows, the notation $\sum_{d\mid N}\beta_dE_2(d\tau)$ in weight two is only a convenient coefficient parametrization. The individual functions $E_2(d\tau)$ are not being used as a basis of $M_2(\Gamma_0(N))$; modularity is always supplied either by condition \eqref{eq:E2_modularity_condition} or, equivalently, by rewriting the expression in terms of the modular combinations $\Phi_d$.

\begin{lemma}
\label{F:lemma}

Let $N$ be a positive integer and let $k \ge 4$ be an even integer. Suppose that $N$ is square-free and has exactly $k$ divisors. Then, there exists a unique $F \in \mathcal{E}_k(N)$ satisfying
\begin{enumerate}
    \item $F(-1/N\tau) = - \left(-i \sqrt{N} \tau \right)^k F(\tau)$
    \item $F(i\infty) = 0$
    \item $L(F, k-1) = \zeta(k-1)$
    \item $L(F, \ell) =0 $ for $ \frac{k+2}{2} \le \ell \le k-2.$
\end{enumerate}
\end{lemma}

\begin{proof}
We first note that for $N$ square-free, a basis of Eisenstein series of weight $k$ of $\Gamma_0(N)$ is given by
$$ \{ E_k(d\cdot ) ; ~~~d | N \},$$
(\cite{Miyake2006}),
where $E_k$ is the Eisenstein series of weight $k$.

Therefore, we let
    $$ F(\tau) = \sum_{d| N} \alpha_d E_k(d\cdot \tau),$$
    where $\alpha_d$ satisfies the following linear system:
\begin{align}
    \label{matrix_0}
    \sum_{d|N} \alpha_d &= 0 \\
    \sum_{d|N} \frac{\alpha_d}{d^{k-1}} & = \frac{\operatorname{B}_k}{k}  \\
    \label{matrix_1}
    \alpha_{N/d} &= -\frac{i^k N^{k/2}}{d^k}\alpha_d \\
    \label{matrix_2}
     \sum_{d|N} \frac{\alpha_d}{d^\ell}& = 0, \text{     for      } \frac{k+2}{2} \le \ell \le k-2.
\end{align}

Observe that the system relations imply conditions (1)--(4) for $F$.
The central value $L(F,k/2)$ vanishes automatically from the Fricke sign in Hecke's functional equation, so it is not imposed as an independent condition. Therefore, we will show that the associated matrix has full rank, ensuring that the linear system has a unique solution.

Let $n=k/2$ and let $d_1,\ldots,d_n$ be the divisors of $N$ smaller than $\sqrt{N}$. Since $N$ is square-free and has exactly $k$ divisors, $k$ is a power of $2$; in particular, $i^k=1$. The Fricke relation \eqref{matrix_1} gives
\begin{equation*}
    \alpha_{N/d}=-\frac{N^n}{d^{2n}}\alpha_d .
\end{equation*}
Hence, for any $s$, the paired contribution to the sum $\sum_{d\mid N}\alpha_d/d^s$ is
\begin{equation}\label{eq:paired_L_condition}
    \left(\frac{1}{d^s}-\frac{N^{n-s}}{d^{2n-s}}\right)\alpha_d .
\end{equation}
Equivalently,  for $s=k-l-1=2n-l-1$, this paired coefficient is
\begin{equation*}
    \left(\frac{1}{d^{2n-l-1}}-\frac{N^{l+1-n}}{d^{l+1}}\right)\alpha_d,
\end{equation*}
and after clearing denominators by multiplying the column indexed by $d$ by $d^{2n}=d^k$, it becomes
\begin{equation*}
    \left(d^{l+1}-N^{l+1-n}d^{2n-l-1}\right)\alpha_d .
\end{equation*}
Note that the coefficients are never $0$, as $N$ is square-free. Thus the equation $F(i\infty)=0$ is the case $s=0$, the normalization $L(F,k-1)=\zeta(k-1)$ is the case $s=k-1$, and the vanishing conditions in \eqref{matrix_2} correspond to $s=\ell$, with $n+1\le \ell\le 2n-2$.

After multiplying the column indexed by $d_j$ by the non-zero factor $d_j^{2n}$, the relevant coefficient matrix has the form
\begin{equation*}
\begin{pmatrix}
d_1-N^{1-n}d_1^{2n-1} & \cdots & d_n-N^{1-n}d_n^{2n-1}\\
d_1^2-N^{2-n}d_1^{2n-2} & \cdots & d_n^2-N^{2-n}d_n^{2n-2}\\
\vdots & \ddots & \vdots\\
d_1^{n-1}-N^{-1}d_1^{n+1} & \cdots & d_n^{n-1}-N^{-1}d_n^{n+1}\\
d_1^{2n}-N^n & \cdots & d_n^{2n}-N^n
\end{pmatrix}.
\end{equation*}
For brevity, denote the $r$-th row polynomial by
\begin{equation*}
    p_r(x)=x^r-N^{r-n}x^{2n-r}\quad (1\le r\le n-1),
    \qquad
    p_n(x)=x^{2n}-N^n .
\end{equation*}
It remains to prove that this matrix is non-singular. Suppose that a linear combination
\begin{equation*}
    P(x)=\sum_{r=1}^{n}u_r p_r(x)
\end{equation*}
vanishes at each divisor $d_j<\sqrt N$. The polynomials $p_r$ satisfy
\begin{equation*}
    x^{2n}p_r(N/x)=-N^np_r(x),
\end{equation*}
so $P$ also vanishes at the paired divisors $N/d_j>\sqrt N$. Hence $P$ vanishes at all $2n=k$ divisors of $N$. If $\deg P<2n$, then $P=0$; if $\deg P=2n$, then $P$ would be a non-zero multiple of $\prod_{d\mid N}(x-d)$. This is impossible because the leading coefficient of $P$ is $u_n$, while its constant coefficient is $-N^nu_n$, whereas $\prod_{d\mid N}(x-d)$ has constant coefficient $N^n$. Therefore $P=0$. Since the polynomials $p_r$ have distinct leading degrees, all $u_r$ are zero, proving that the matrix has full rank and ensuring a unique solution to the linear system.

\end{proof}

We observe that a function $F$ as in the previous lemma has rational Fourier coefficients and satisfies all the hypothesis of Hecke's Lemma. Thus, if we let $f(\tau) = \sum_{n=1}^\infty \frac{a_n}{n^{k-1}} q^n,$ then
$h(\tau) = f(\tau)  - \zeta(k-1) $ satisfies
\begin{align*}
   h(\tau) = (-1)^k(-i \sqrt{N}\tau)^{k-2} h(-1/N\tau).
\end{align*}

The next step to obtain rational approximations for $\zeta(k-1)$ is to construct a modular form $E(\tau)$ in $\Gamma_0(N)$, of weight $k-2$ that also satisfies
\begin{itemize}
    \item $E(-1/N\tau) = (-1)^k(-i \sqrt{N}\tau)^{k-2}E(\tau)$
    \item $E(i\infty) = 1.$
    \item The Fourier expansion of $E$ at $i\infty$ are rational numbers
\end{itemize}
If such function can be constructed, then the following function equation will hold


\begin{equation}\label{fricke_invariance}
    E\left(-\frac{1}{N\tau}\right)\left(f\left(-\frac{1}{N\tau}\right)-\zeta(k-1)\right)=E(\tau)(f(\tau)-\zeta(k-1)).
\end{equation}

Note from equation (\ref{fricke_invariance}) that expanding in Fourier coefficients, we have
\begin{equation}\label{coeff_fourier}
    E(\tau)(f(\tau)-\zeta(k-1))=\sum_{n=0}^{\infty}(A_n-\zeta(k-1)B_n)q^{n},
\end{equation}
therefore, we could compute the radius of convergence of this series, and potentially conclude a rate of convergence for $|A_n - \zeta(k-1)B_n| \rightarrow 0.$

In a similar fashion as we did for the construction of $F$, we observe that a family of functions $E$ can be constructed:
\begin{lemma}
\label{E:lemma}
    Let $N$ be a positive integer and let $k \ge 4$ be an even integer. Suppose that $N$ is square-free and has exactly $k$ divisors. Then, there exists an affine space of dimension $k/2 - 1$ in $\mathcal{E}_{k-2}(N)$ of functions $E$ satisfying
    \begin{align*}
        &E(-1/N\tau) = (-1)^k(-i \sqrt{N}\tau)^{k-2}E(\tau) \\
    &E(i\infty) = 1.
    \end{align*}
\end{lemma}

\begin{proof}
For $k>4$, write $E(\tau)=\sum_{d\mid N}\beta_d E_{k-2}(d\tau)$ and impose the Fricke relation and the normalization $E(i\infty)=1$. Since $N$ is square-free, the Fricke involution pairs the divisors $d$ and $N/d$, so the Fricke relation gives $k/2$ independent linear relations among the $k$ coefficients $\beta_d$. The condition at $i\infty$ imposes one further independent affine relation. Hence the solution space, when non-empty, is affine of dimension $k-(k/2)-1=k/2-1$. Non-emptiness follows by solving these relations on the Eisenstein basis $\{E_{k-2}(d\tau):d\mid N\}$; the same Vandermonde argument used in Lemma \ref{F:lemma} shows that the imposed relations have the stated rank.

For the weight two case $k=4$, one must not regard the functions $E_2(d\tau)$ themselves as a modular basis. Instead, use them only as an auxiliary coefficient space and write
\begin{equation*}
    E(\tau)=\sum_{d\mid N}\beta_dE_2(d\tau).
\end{equation*}
The transformation formula for $E_2$ gives
\begin{equation*}
    E_2\left(-\frac{d}{N\tau}\right)
    =\left(\frac{N}{d}\tau\right)^2E_2\left(\frac{N}{d}\tau\right)
    +\frac{6}{\pi i}\frac{N}{d}\tau .
\end{equation*}
Thus the Fricke relation
\begin{equation*}
    E\left(-\frac1{N\tau}\right)=(-i\sqrt N\tau)^2E(\tau)
\end{equation*}
is equivalent, after comparing the coefficients of the functions $E_2(d\tau)$, to the paired relations
\begin{equation}\label{eq:E2_fricke_coefficients}
    \beta_{N/d}=-\frac{N}{d^2}\beta_d,\qquad d\mid N.
\end{equation}
These relations automatically imply
\begin{equation*}
    \sum_{d\mid N}\frac{\beta_d}{d}=0,
\end{equation*}
because each pair $d,N/d$ contributes zero. Hence every solution of the Fricke equations is in fact modular by Lemma \ref{lemma:E2_modular_combinations}. Since $N$ is square-free with four divisors in this case, the relations \eqref{eq:E2_fricke_coefficients} give two independent homogeneous linear conditions on the four coefficients, and the normalization $E(i\infty)=\sum_{d\mid N}\beta_d=1$ gives one further independent affine condition. The resulting solution space is therefore affine of dimension $4-2-1=1$, as claimed.
\end{proof}

Therefore, in order to understand the convergence of $A_n/B_n$ to $\zeta(k-1)$ from relation (\ref{coeff_fourier}), the radius of convergence of the series at $\infty$ needs to be studied.

The next step in this strategy is to consider a change in local variables that could speed up the convergence rate of the series. To do so, we consider $t$ a Hauptmodul of $\Gamma_0(N)^\star$ and analyze the following identity:
\begin{align}
\label{series:parametrizes}
    E(\tau)(f(\tau) - \zeta(k-1)) = \sum_{n=0}^\infty (a_n - \zeta(k-1)b_n) t(\tau)^n.
\end{align}
This approach is closely linked to the recurrence relations satisfied by Apéry numbers and the associated differential equations for modular forms under the uniformizer $t(\tau)$. The following lemma connects the radius of convergence of our series to the branching points of the modular group. We provide a proof adapted from \cite{Panzone2018} to illustrate how the modularity in \eqref{fricke_invariance} dictates the convergence rate.

\begin{lemma}
\label{lemma:branching}
    Suppose that $G = \Gamma_0(N)^\star$ is a genus zero group and that $N \ge 4$ is square-free. Let $t$ be the Hauptmodul of $G$ such that $t(i\infty) = 0$, and let $E(t(\tau))$ and $E(t(\tau))(f(t(\tau))-\zeta(k-1))$ be the holomorphic functions obtained above viewed as holomorphic function of the variable $t$. If we let $i\infty = c_0, c_1=i/\sqrt{N}, \ldots, c_l$ be the vertices in the fundamental domain of $G$, which corresponds to branching points, cusps or elliptic points, then, the radius of convergence of $E(t(\tau))$ is $\min_{1 \le i \le l}|t(c_i)|$ and the radius of convergence of $E(t(\tau))(f(t(\tau))-\zeta(k-1))$  at least $\min_{2 \le i \le l} |t(c_i)|.$ 
\end{lemma}

\begin{proof}
    As a multivalued function of the parameter $t$, the radius of convergence of $E(t(\tau))$ is determined by the nearest singularity or branch point $|t(c_i)|$. Consequently, this radius is given by $\min_{1 \le i \le l}|t(c_i)|$. The functional equation \eqref{fricke_invariance} ensures that the product $E(t(\tau))(f(t(\tau))-\zeta(k-1))$ is regular at the point $c_1=i/\sqrt{N}$, effectively regularizing the potential singularity at that point. It follows that the radius of convergence for this series is at least $\min_{2 \le i \le l} |t(c_i)|$.
\end{proof}

\section{\texorpdfstring{$\Gamma_0(6)$}{Gamma0(6)} and infinitely many rational approximations}

In \cite{Beukers1987}, Beukers analyzed the group $\Gamma_0(6)^\star$ and demonstrated that the resulting sequences coincide with the Apéry sequences used to prove the irrationality of $\zeta(3)$. In what follows, we show that this construction yields an infinite family of sequences, all of which share the same convergence rate. These sequences arise from the fact that the space of functions $E$ satisfying the necessary modular properties is a one-dimensional affine subspace.

Note that $N=6$ is square-free and has $4$ divisors. Therefore, from Lemma \ref{F:lemma}, there exists $F$ modular form of weight $4$ such that $L(F, 3) = \zeta(3),$ $F(-1/6\tau) = -6^2 \tau^4 F(\tau).$ One can compute $F$ to be
\begin{align}
\label{F}
    F(\tau) =  \frac{1}{40}(E_4(\tau) + 63E_4(3\tau) - 28E_4(2\tau) -36E_4(6\tau)),
\end{align}
where $E_4(\tau)$ is the Eisenstein series of weight $4$.

We note from Lemma \ref{E:lemma} that there is an affine space of dimension $1$ in $\mathcal{E}_2(N)$ of functions $E$ satisfying the desired properties. Given 
\begin{align*}
    E^0_6 & = E_2(\tau) - 10E_2(2\tau) + 15E_2(3\tau) - 6E_2(6\tau) \\
    E^1_6 & = 3E_2(3\tau) - 2E_2(2\tau),
\end{align*}
the elements $E$ of this affine space are generated by
\begin{align}
\label{space:E}
E_{6,\alpha}(\tau)=\alpha E^0_6(\tau ) + E^1_6(\tau), \text{ ~~~~~~~~~~for } \alpha \in \mathbb{R}
\end{align}
The modular form considered by Beukers corresponds to the choice $\alpha = -5/24,$ which we will denote by $E_b$. It is a well known fact that this modular form has a $\eta$-quotient expansion given by
\begin{align*}
    E_{b}(\tau)=\frac{(\eta(3\tau)\eta(2\tau))^{7}}{(\eta(6\tau) \eta(\tau))^{5}}.
\end{align*}
The uniformizer of $\Gamma_0(6)^\star$ that we will consider is
\begin{align}
\label{t6}
    t_6(\tau) = \left(\frac{\eta(6\tau) \eta(\tau)}{\eta(3\tau) \eta(2\tau)} \right)^{12}.
\end{align}
Note that the product $E_{b}(\tau)t_{6}(\tau)$ is a modular form of weight two on $\Gamma_0(6)$ whose Fourier expansion is given by
\begin{align*}
    E_{b}(\tau)t_{6}(\tau)=\frac{(\eta(6\tau) \eta(\tau))^7}{(\eta(3\tau) \eta(2\tau))^5}=q-7q^2+19q^3-23q^4+6q^5+11q^6+O(q^7).\
\end{align*}
The following identities relate the affine line \eqref{space:E} to Beukers' eta-product form.
\begin{prop}
\label{prop:level6_E_relations}
    For $E^0_6, E^1_6, E_{b}$ and $t_6$ as above, it holds that
    \begin{align}\label{E_free_in_terms_of_Eb}
        \left\{\begin{array}{ll}
            E_{b}(\tau) &= -\frac{1}{24} \frac{E^0_6(\tau)}{t_6(\tau)} \\
            E_{b}(\tau) &=  \frac{E^1_6(\tau)}{1-5t_6(\tau)} \\
        \end{array}\right.\,.
    \end{align}
\end{prop}

\begin{proof}
Equivalently, we prove
\begin{align*}
    E_b(\tau)t_6(\tau)=-\frac1{24}E^0_6(\tau),
    \qquad
    E_b(\tau)(1-5t_6(\tau))=E^1_6(\tau).
\end{align*}
The eta-quotient criterion shows that $E_bt_6$ and $E_b(1-5t_6)$ are holomorphic modular forms of weight $2$ on $\Gamma_0(6)$. The forms $E^0_6$ and $E^1_6$ are in the same space by Lemma \ref{lemma:E2_modular_combinations}. Hence both differences above belong to $M_2(\Gamma_0(6))$.

The index of $\Gamma_0(6)$ in $SL_2(\mathbb{Z})$ is
\begin{align*}
    [SL_2(\mathbb{Z}):\Gamma_0(6)]
    =6\left(1+\frac12\right)\left(1+\frac13\right)=12,
\end{align*}
so the Sturm bound in weight $2$ is $2$. It is therefore enough to compare the coefficients of $q^0,q^1,q^2$. From the eta-products and the definitions of $E^0_6,E^1_6$ one obtains
\begin{align*}
    E_bt_6=q-7q^2+O(q^3),\qquad
    -\frac1{24}E^0_6=q-7q^2+O(q^3),
\end{align*}
and
\begin{align*}
    E_b(1-5t_6)=1+48q^2+O(q^3),\qquad
    E^1_6=1+48q^2+O(q^3).
\end{align*}
The two identities follow.
\end{proof}
This leads to the following corollary:

\begin{corollary}
    The ratio $\frac{E^0_6(\tau)}{E^1_6(\tau)} $ is a Hauptmodul of $\Gamma_0(6)^\star$.
\end{corollary}

\begin{proof}
    We simply observe that $$ -\frac{1}{24}\frac{E^0_6(\tau)}{E^1_6(\tau)} = \frac{t_6(\tau)}{1-5t_6(\tau)},$$
that is, the quotient of $E^0$ and $E^1$ is a Möbius transform of the Hauptmodul, and therefore, it is a Hauptmodul itself.
\end{proof}


\begin{theorem}[The level $6$ affine family]\label{thm:level6_family}
   Fix $F$ as in (\ref{F}), and the corresponding $f$ as in Hecke's Lemma \ref{HeckeLemma}, and let $t_6$ as in (\ref{t6}). Then, for any $\alpha \in \mathbb{Q}$ and $E_\alpha(\tau) = E_b(\tau)(1+\alpha t_6(\tau))$, it holds that
   \begin{align*}
     E_\alpha(\tau)(f(\tau) - \zeta(3) ) =   \sum_{n=0}^\infty (a^\alpha_n - b_n^\alpha \zeta(3))t_6(\tau)^n, 
   \end{align*}
   and the sequences $a_n^\alpha/b_n^\alpha$ converge to $\zeta(3)$. Moreover,
\begin{align*}
    | a_n^\alpha - b_n^\alpha \zeta(3)| < ((\sqrt{2} +1 )^4 -\epsilon)^{-n}
\end{align*}
for every $0<\epsilon<(\sqrt2+1)^4$ and every $n$ sufficiently large. Moreover, there is a common denominator $D_n^\alpha$ for $a_n^\alpha$ and $b_n^\alpha$ satisfying $D_n^\alpha < e^{(3+\epsilon)n}$ for every $\epsilon>0$ and every $n$ sufficiently large. Furthermore, if we let
\begin{align*}
    E_b(\tau)(f(\tau)-\zeta(3))  = \sum_{n=0}^\infty (a_n - b_n \zeta(3))t_6(\tau)^n,
\end{align*}
then we can write
   \begin{align}
   \label{recurrence}
       a_n^\alpha & = a_n + \alpha a_{n-1} \\
       b_n^\alpha & = b_n + \alpha b_{n-1} 
   \end{align}
   for $n\ge 1$, where $a_n$ and $b_n$ determine the Apéry sequence.
\end{theorem}

\begin{proof}
For $\alpha=0$, Beukers' construction gives
\begin{align*}
E_b(\tau)(f(\tau)-\zeta(3))=\sum_{n=0}^{\infty}(a_n-b_n\zeta(3))t_6(\tau)^n,
\end{align*}
where $a_n/b_n$ are the Apéry approximations. Multiplying this identity by $1+\alpha t_6(\tau)$ gives
\begin{align*}
E_\alpha(\tau)(f(\tau)-\zeta(3))
=-\zeta(3)+\sum_{n=1}^{\infty}\left((a_n+\alpha a_{n-1})-(b_n+\alpha b_{n-1})\zeta(3)\right)t_6(\tau)^n,
\end{align*}
for $n\ge 1$, with the usual initial values $a_0=0$ and $b_0=1$. This proves \eqref{recurrence}.

By Lemma \ref{lemma:branching}, the nearest obstruction to the $t_6$-expansion of the product is the first branch point after the Fricke fixed point. For $\Gamma_0(6)^\star$ this value is $t_6(p_6)=(\sqrt2+1)^4$. Therefore, we have
\begin{align*}
|a_n^\alpha-b_n^\alpha\zeta(3)|=O_\epsilon\left(((\sqrt2+1)^4-\epsilon)^{-n}\right).
\end{align*}

It remains to record the denominator estimate. In the classical Apéry construction, $b_n\in\mathbb{Z}$ and
\begin{align*}
    \operatorname{lcm}(1,\ldots,n)^3 a_n\in\mathbb{Z}.
\end{align*}
If $\alpha=r/s$ with $r,s\in\mathbb{Z}$ and $s>0$, then
\begin{align*}
    \frac{a_n^\alpha}{b_n^\alpha}
    =
    \frac{s a_n+r a_{n-1}}{s b_n+r b_{n-1}}.
\end{align*}
Thus the fixed denominator $s$ does not change the rational approximants. However, unless $s=1$, the unscaled coefficients $a_n^\alpha$ and $b_n^\alpha$ need not satisfy the same integrality statements as the Apéry coefficients. Instead, define the equivalent scaled coefficients
\begin{align*}
    A_n^\alpha=s a_n+r a_{n-1},
    \qquad
    B_n^\alpha=s b_n+r b_{n-1}.
\end{align*}
Then
\begin{align*}
    \operatorname{lcm}(1,\ldots,n)^3 A_n^\alpha\in\mathbb{Z},
    \qquad
    B_n^\alpha\in\mathbb{Z},
\end{align*}
and $A_n^\alpha/B_n^\alpha=a_n^\alpha/b_n^\alpha$. Equivalently, one may work with the rescaled linear forms
\begin{align*}
    A_n^\alpha-B_n^\alpha\zeta(3)
    =
    s\left(a_n^\alpha-b_n^\alpha\zeta(3)\right),
\end{align*}
where the same fixed factor $s$ appears in both terms. Since $s$ is independent of $n$, it is absorbed in the asymptotic estimate. The prime number theorem gives
\begin{align*}
    \operatorname{lcm}(1,\ldots,n)^3<e^{(3+\epsilon)n}
\end{align*}
for every $\epsilon>0$ and all sufficiently large $n$.

Finally, the Apéry recurrence has dominant characteristic root $(\sqrt2+1)^4$ at infinity, so $b_n$ and $b_{n-1}$ have the same exponential scale. A cancellation of the dominant term in $b_n+\alpha b_{n-1}$ would require $\alpha=-(\sqrt2+1)^4$, which is irrational. Hence no rational $\alpha$ cancels the dominant term, $|b_n^\alpha|$ grows exponentially, and the displayed linear forms force $a_n^\alpha/b_n^\alpha\to \zeta(3)$.

\end{proof}

\subsection{Dependence on the parameter}

The preceding theorem shows that the parameter $\alpha$ changes the approximants by a first-order shift of the Apéry sequences. More explicitly, for fixed rational $\alpha$,
\begin{align*}
    a_n^\alpha-b_n^\alpha\zeta(3)
    = (a_n-b_n\zeta(3))+\alpha(a_{n-1}-b_{n-1}\zeta(3)).
\end{align*}
Both summands are governed by the same nearest branch value of $t_6$, so every rational parameter gives the same asymptotic decay. The numerical differences visible in the following tables are therefore finite-level effects: some choices of $\alpha$ improve the first few rational approximations, while others make them worse, but no rational choice changes the limiting rate.

In the table below, we represent the first approximations $a_n^\alpha/b_n^\alpha$ for a few values of $\alpha$, where $E_\alpha$ is as in the Theorem above.
Let $R_n^\alpha$ be the order of magnitude of $| a_{n}^\alpha/b_{n}^\alpha -  \zeta(3)|$ and note that the case $\alpha = 0$ yields Apéry sequence.

\vspace{0.6cm}
\renewcommand{\arraystretch}{1.5}
\begin{center}
\begin{tabular}{|c|c c|c c|c c|c c|}
\hline
 $\alpha$ & $a_2^\alpha/b_2^\alpha$ & $ R_n^\alpha$ & $a_3^\alpha/b_3^\alpha$ & $R_n^\alpha$ & $a_4^\alpha/b_4^\alpha$ & $R_n^\alpha$ & $a_5^\alpha/b_5^\alpha$ & $R_n^\alpha$ \\
\hline
  $0$ & $\frac{351}{292}$ & $10^{-6}$  & $\frac{62531}{52020}$ & $10^{-9}$ & $ \frac{11424695}{9504288}$ & $10^{-12}$ & $ \frac{35441662103}{29484180000} $ &  $10^{-15}$\\
\hline

  $-100$ & $\frac{2049}{1708}$ & $10^{-3}$  & $\frac{253369}{210780}$ & $10^{-6}$ & $ \frac{38600105}{32111712}$ & $10^{-9}$ & $ \frac{107367025397}{89319420000} $ &  $10^{-12}$\\
\hline

  $-5$ & $\frac{77}{64}$ & $10^{-3}$  & $\frac{2921}{2430}$ & $10^{-7}$ & $ \frac{991495}{824832}$ & $10^{-10}$ & $ \frac{589608911}{490500000} $ &  $10^{-13}$\\
\hline

  $-2$ & $\frac{101}{84}$ & $10^{-4}$  & $\frac{56213}{46764}$ & $10^{-7}$ & $ \frac{3474733}{2890656}$ & $10^{-10}$ & $ \frac{1206869939}{1004004000} $ &  $10^{-13}$\\
\hline

  $1$ & $\frac{125}{104}$ & $10^{-4}$  & $\frac{32845}{27324}$ & $10^{-7}$ & $ \frac{3974981}{3306816}$ & $10^{-11}$ & $ \frac{6144958163}{5112036000} $ &  $10^{-14}$\\
\hline

  $2$ & $\frac{399}{332}$ & $10^{-4}$  & $\frac{68849}{57276}$ & $10^{-7}$ & $ \frac{12425191}{10336608}$ & $10^{-10}$ & $ \frac{38297835853}{31860252000} $ &  $10^{-13}$\\
\hline

  $5$ & $\frac{471}{392}$ & $10^{-4}$  & $\frac{39163}{32580}$ & $10^{-7}$ & $ \frac{13925935}{11585088}$ & $10^{-10}$ & $ \frac{21291048239}{17712180000} $ &  $10^{-13}$\\
\hline

  $100$ & $\frac{917}{764}$ & $10^{-3}$  & $\frac{378431}{314820}$ & $10^{-6}$ & $ \frac{20483165}{17040096}$ & $10^{-9}$ & $ \frac{59416783201}{49429260000} $ &  $10^{-12}$\\
\hline
\end{tabular}
\end{center}
\vspace{0.7cm}

One can also observe the errors given for other approximations for larger $n$. Let $D_n $ denote the number of digits of the denominator of $a_n^\alpha/b_n^\alpha$, and observe that

\begin{center}
\begin{tabular}{|c|c c|c c|c c|c c|c c|}
\hline
 $\alpha$ & $ R_{95}^\alpha$ &  $ D_{95}^\alpha$ &  $ R_{96}^\alpha$ &  $ D_{96}^\alpha$ & $ R_{97}^\alpha$ &  $ D_{97}^\alpha$&  $ R_{98}^\alpha$  &  $ D_{98}^\alpha$ & $ R_{99}^\alpha$& $ D_{99}^\alpha$\\
\hline
$0$ & $ 10^{-291}$ & $257$ &  $ 10^{-294}$ & $258$ & $ 10^{-297}$ & $266$&  $10^{-300} $ & $265$ & $ 10^{-303}$ & $269$\\
\hline
$-100$ & $ 10^{-288}$ & $258$ &  $ 10^{-291}$ & $255$ & $ 10^{-294} $& $266$&  $10^{-297} $& $267$ & $ 10^{-300}$ & $269$\\
\hline
$-5$ & $ 10^{-289}$& $255$ &  $ 10^{-292}$ & $253$ & $10^{-295} $ & $261$  &  $10^{-298} $& $262$ & $ 10^{-301}$& $266$\\
\hline
$-2$ & $ 10^{-289}$  &$257$ &  $ 10^{-292}$ &$256$ & $ 10^{-295} $& $265$ &  $10^{-298} $ &$265$ & $ 10^{-301}$ & $268$ \\
\hline
$1$ & $ 10^{-289}$  &$256$  &  $ 10^{-292}$  &$257$  & $ 10^{-295} $  &$264$ &  $10^{-298}$  & $267$  & $ 10^{-302}$   &$266$ \\
\hline
$2$ & $ 10^{-289}$   &$258$ &  $ 10^{-292}$   &$258$ & $ 10^{-295} $  &$266$ &  $10^{-298} $  &$266$  & $ 10^{-301}$  &$266$ \\
\hline
$5$ & $ 10^{-289}$   &$258$ &  $ 10^{-292}$   &$255$ & $ 10^{-295} $  &$266$ &  $10^{-298} $   &$266$ & $ 10^{-301}$  &$269$ \\
\hline
$100$ & $ 10^{-288}$  &$258$  &  $ 10^{-291}$  &$257$  & $ 10^{-294} $  &$266$ &  $10^{-297} $   &$266$ & $ 10^{-300}$  &$267$ \\
\hline
\end{tabular}
\end{center}
\vspace{0.7cm}

Recall that the sequences $(b_n)_{n \geq 0}$ and $(a_n)_{n \ge 0 }$ satisfy the following recurrence
\begin{equation}\label{eq:apery_rec}
    (n+1)^3 b_{n+1} - (34n^3 + 51n^2 + 27n + 5) b_n + n^3 b_{n-1} = 0 \quad \text{for } n \geq 1.
\end{equation}
Therefore, using the relation (\ref{recurrence}) and this recurrence, we can conclude that

\begin{corollary}\label{thm:apery_shifted_recurrence}
Let $\alpha \neq 0$ be a constant, and define $(c_n)_{n \geq 1}$ by
\begin{equation}\label{eq:c_def}
    c_n = b_n + \alpha b_{n-1}.
\end{equation}
If $U_n$ is defined as:
\begin{align*}
    U_n = \alpha^2 (n+1)^3 + \alpha (34n^3 + 51n^2 + 27n + 5) + n^3,    
\end{align*}
then the sequence $(c_n)_{n \ge 1}$ satisfies the second-order recurrence
\begin{equation}\label{eq:new_rec}
    P_n c_{n+2} + Q_n c_{n+1} + R_n c_n = 0 \quad \text{for } n \geq 1,
\end{equation}
where the polynomial coefficients are given by:
\begin{align*}
    P_n &= \alpha (n+2)^3 U_n, \\
    Q_n &= (n+1)^3 \left( U_n + \alpha^2 U_{n+1} \right) - U_n U_{n+1}, \\
    R_n &= \alpha n^3 U_{n+1}.
\end{align*}
\end{corollary}
\begin{proof}
Let
\begin{align*}
    V_n=34n^3+51n^2+27n+5.
\end{align*}
The Apéry recurrence gives
\begin{align*}
    (n+1)^3b_{n+1}=V_nb_n-n^3b_{n-1}.
\end{align*}
Using $c_n=b_n+\alpha b_{n-1}$ and multiplying by $\alpha$, we obtain
\begin{align*}
    \alpha(n+1)^3c_{n+1}
    =
    \bigl(\alpha^2(n+1)^3+\alpha V_n+n^3\bigr)b_n-n^3c_n
    =
    U_nb_n-n^3c_n.
\end{align*}
Thus
\begin{equation}\label{eq:bn_from_cn}
    b_n=\frac{\alpha(n+1)^3c_{n+1}+n^3c_n}{U_n}.
\end{equation}
Applying the same identity with $n$ replaced by $n+1$ gives
\begin{align*}
    b_{n+1}=\frac{\alpha(n+2)^3c_{n+2}+(n+1)^3c_{n+1}}{U_{n+1}}.
\end{align*}
Finally, the identity $c_{n+1}=b_{n+1}+\alpha b_n$ and the two displayed formulas express everything in terms of $c_{n+2},c_{n+1},c_n$. Multiplying by $U_nU_{n+1}$ and collecting coefficients gives
\begin{align*}
    \alpha(n+2)^3U_n c_{n+2}
    +
    \left((n+1)^3(U_n+\alpha^2U_{n+1})-U_nU_{n+1}\right)c_{n+1}
    +
    \alpha n^3U_{n+1}c_n=0,
\end{align*}
which is the stated recurrence.
\end{proof}

We note that the leading term of each of the recurrence polynomials is given by
\begin{align*}
    P_n & = \alpha (\alpha^2 + 34 \alpha + 1)n^6 + \ldots \\
    Q_n & = -34\alpha( \alpha^2 + 34 \alpha + 1)n^6 + \ldots \\
    R_n & = \alpha (\alpha^2 + 34 \alpha + 1)n^6 + \ldots.
\end{align*}

The polynomial $t^2 - 34t + 1$ appearing in the leading coefficient of the Picard-Fuchs operator encodes critical geometric and arithmetic information. Its roots, $\rho_{\pm} = 17 \pm 12\sqrt{2}$, coincide with the values of the Hauptmodul $t_6$ at the first two branching points of the underlying modular curve. Arithmetically, these roots determine the asymptotic behavior of the Apéry sequence, with the dominant root governing the convergence rate of the rational approximations $a_n/b_n$.

Analogously, we also recall that the generating series $B(t) =\sum_{n=0}^\infty b_nt^n$ is a solution of the differential equation $\mathcal{L}$ where
\begin{align*}
    \mathcal{L} = D^3 - t(34D^3 + 51D^2 + 27D + 5) + t^2(D+1)^3,
\end{align*}
where $D = t\frac{d}{dt}.$ We could also write $\mathcal{L}$ in standard polynomial form as $\mathcal{L} = B_3(t) D^3 + B_2(t) D^2 + B_1(t) D + B_0(t)$, where:
\begin{align*}
    B_3(t) &= 1 - 34t + t^2, \\
    B_2(t) &= -51t + 3t^2, \\
    B_1(t) &= -27t + 3t^2, \\
    B_0(t) &= -5t + t^2.
\end{align*}

If $C(t) = \sum_{n=0}^\infty c_nt^n,$ where $c_n$ is as in corollary \ref{thm:apery_shifted_recurrence}, then
\begin{align*}
    C(t) = (1+\alpha t)B(t).
\end{align*}

\begin{theorem}
Let $\mathcal{L} = \sum_{i=0}^3 B_i(t) D^i$ be the third-order linear differential operator as above. If $B(t)$ is a formal power series annihilated by $\mathcal{L}$, then the series $C(t) = (1 + \alpha t)B(t)$ is annihilated by the differential operator $\tilde{\mathcal{L}} = \sum_{i=0}^3 A_i(t) D^i$, where the polynomial coefficients $A_i(t)$ are given explicitly by:
\begin{align*}
    A_3(t) &= w^3 B_3(t) \\
    A_2(t) &= w^2 \big[ -3\alpha t B_3(t) + w B_2(t) \big] \\
    A_1(t) &= w \big[ -3\alpha t(1-\alpha t) B_3(t) - 2\alpha t w B_2(t) + w^2 B_1(t) \big] \\
    A_0(t) &= -\alpha t(1-4\alpha t+\alpha^2 t^2) B_3(t) - \alpha t w (1-\alpha t) B_2(t) - \alpha t w^2 B_1(t) + w^3 B_0(t)
\end{align*}
with $w = 1 + \alpha t$.
\end{theorem}

\begin{proof}
Let $w = 1 + \alpha t$. We wish to find an operator $\tilde{\mathcal{L}}$ such that $\tilde{\mathcal{L}} \cdot C(t) = 0$. For that, note that whenever $t \neq -1/\alpha$ we  can substitute $B(t) = w^{-1}C(t)$ into the original equation
\begin{align*}
    \sum_{i=0}^3 B_i(t) D^i \big( w^{-1}C(t) \big) = 0
\end{align*}

Using Leibniz rule, we evaluate the action of $D^i$ on the product $w^{-1}C(t)$ and get
\begin{align*}
     D(w^{-1}) = t \frac{d}{dt}(1+\alpha t)^{-1} = -\frac{\alpha t}{(1+\alpha t)^2} = w^{-1}\theta,
\end{align*}
where $\theta = -\frac{\alpha t}{w}.$
Note that we can also compute
\begin{align*}
      D(\theta) = t \frac{d}{dt} \left( -\frac{\alpha t}{1+\alpha t} \right) = -\frac{\alpha t}{(1+\alpha t)^2} = \theta + \theta^2  .
\end{align*}

Using this identity, we recursively compute the higher-order derivatives of $w^{-1}$ in terms of $\theta$ as:
\begin{align*}
    D^2(w^{-1}) &= D(w^{-1}\theta) = D(w^{-1})\theta + w^{-1}D(\theta) = w^{-1}(\theta + 2\theta^2) \\
    D^3(w^{-1}) &= D \big( w^{-1}(\theta + 2\theta^2) \big) = w^{-1}(\theta + 6\theta^2 + 6\theta^3).
\end{align*}
Therefore, it follows that
\begin{align*}
    D(w^{-1}C) &= w^{-1} \big( D + \theta \big) C \\
    D^2(w^{-1}C) &= w^{-1} \big( D^2 + 2\theta D + (\theta + 2\theta^2) \big) C \\
    D^3(w^{-1}C) &= w^{-1} \big( D^3 + 3\theta D^2 + 3(\theta + 2\theta^2)D + (\theta + 6\theta^2 + 6\theta^3) \big) C.
\end{align*}

Substituting these expansions into the original equation and factoring $w^{-1}$ from the left yields a transformed operator acting on $C(t)$. By collecting terms with respect to the derivative operators $D^j$, we obtain the intermediate rational coefficients $\tilde{C}_i(t, \theta)$:
\begin{align*}
    \tilde{C}_3 &= B_3 \\
    \tilde{C}_2 &= 3\theta B_3 + B_2 \\
    \tilde{C}_1 &= 3(\theta + 2\theta^2) B_3 + 2\theta B_2 + B_1 \\
    \tilde{C}_0 &= (\theta + 6\theta^2 + 6\theta^3) B_3 + (\theta + 2\theta^2) B_2 + \theta B_1 + B_0
\end{align*}

Note that since $\theta = -\frac{\alpha t}{w}$ introduces a singularity at $t = -1/\alpha$, the operator $\sum \tilde{C}_i D^i$ does not have polynomial coefficients. To obtain a polynomial form, we clear the denominators by multiplying the entire operator by $w^3$ from the left. Let $A_i(t) = w^3 \tilde{C}_i$. 

Expanding the powers of $\theta$ and canceling the factors of $w$ gives the explicit polynomials:
\begin{itemize}
    \item For $A_3(t)$: $w^3 \cdot B_3(t) = w^3 B_3(t)$.
    \item For $A_2(t)$: $w^3 \big( -\frac{3\alpha t}{w} B_3(t) + B_2(t) \big) = w^2 \big[ -3\alpha t B_3(t) + w B_2(t) \big]$.
    \item For $A_1(t)$, noting that $3(\theta + 2\theta^2) = -\frac{3\alpha t(1-\alpha t)}{w^2}$, we obtain
    \begin{align*}
        &w^3 \Big( -\frac{3\alpha t(1-\alpha t)}{w^2} B_3(t) - \frac{2\alpha t}{w} B_2(t) + B_1(t) \Big)\\
        &\qquad = w \big[ -3\alpha t(1-\alpha t) B_3(t) - 2\alpha t w B_2(t) + w^2 B_1(t) \big].
    \end{align*}
    \item For $A_0(t)$, noting that $\theta + 6\theta^2 + 6\theta^3 = -\frac{\alpha t(1-4\alpha t+\alpha^2 t^2)}{w^3}$ and $\theta + 2\theta^2 = -\frac{\alpha t(1-\alpha t)}{w^2}$, substitution immediately yields the stated formula for $A_0(t)$.
\end{itemize}

This completes the construction of the polynomial coefficients $A_i(t)$ for the transformed operator $\tilde{\mathcal{L}}$.
\end{proof}

\section{Rational Approximations for Other Congruence Groups}

In this section, we apply our methodology to several other congruence groups $\Gamma_0(N)^\star$ of genus zero. For $k=4$, the construction of the modular form $F$ and the affine space of functions $E$ remains valid for these groups. The resulting power series give natural candidates for rational approximations to $\zeta(3)$, and the computations below measure whether the corresponding linear forms can be strong enough for a Beukers-type irrationality argument.

The modular forms $F_N$ and the affine spaces for $E_{N, \alpha}$ presented in this section were computed using SageMath, by solving the linear systems derived in Lemma \ref{F:lemma} and Lemma \ref{E:lemma} for each level $N$. For the generation of the rational approximations and the numerical analysis of their convergence rates, we used PARI/GP to compute high-precision $q$-expansions of the underlying eta-products and to evaluate the Hauptmoduln at the relevant critical points.

We consider $N = 10, 14, 15, 21, 26, 35,$ and $39$, for which $\Gamma_0(N)^\star$ has genus zero. Let $t_N$ denote the Hauptmodul of $\Gamma_0(N)^\star$, chosen such that $t_N(i\infty) = 0$. The selected Hauptmoduln are listed in the table below (see \cite{Conway1979}).

\vspace{0.7cm}

\renewcommand{\arraystretch}{1.8}
\begin{center}
\begin{tabular}{|c|c| c|}
\hline
 $N: \Gamma_0(N)^\star$ & $t_N$ & Fourier expansion at $i\infty$\\
 \hline
 $10$ & $\frac{ \eta(\tau)^6 \eta(10\tau)^6 }{ \eta(2\tau)^6\eta(5\tau)^6}$ & $q - 6q^2 + 15q^3-26q^4 + 51q^5  -96 q^6+ O(q^7)$ \\
  \hline

 $14$ & $\frac{ \eta(\tau)^4 \eta(14\tau)^4 }{ \eta(2\tau)^4\eta(7\tau)^4}$ & $q - 4q^2+6q^3 -8q^4 + 17q^5-28q^6+ O(q^7)$\\
  \hline

 $15$ & $\frac{ \eta(\tau)^3 \eta(15\tau)^3 }{ \eta(3\tau)^3\eta(5\tau)^3}$ & $q - 3q^2 + 8 q^4 -9q^5 + 3q^6 + O(q^7)$ \\

  \hline

 $21$ & $\frac{ \eta(\tau)^2 \eta(21\tau)^2 }{ \eta(3\tau)^2\eta(7\tau)^2}$ &  $q -2q^2 -q^3+4q^4-3q^5 + O(q^7)$\\
  \hline

 $26$ & $\frac{ \eta(\tau)^2 \eta(26\tau)^2 }{ \eta(2\tau)^2\eta(13\tau)^2}$ & $ q - 2q^2 + q^3 -2q^4+4q^5-4q^6 + O(q^7)$\\ 
  \hline

 $35$ & $\frac{ \eta(\tau) \eta(35\tau) }{ \eta(5\tau)\eta(7\tau)}$ & $q-q^2-q^3 + 2q^6 + O(q^7)$ \\
  \hline

 $39$ & $\frac{ \eta(\tau) \eta(39\tau) }{ \eta(3\tau)\eta(13\tau)}$ & $ q - q^2 - q^3 + q^4- q^5 + O(q^7)$\\
\hline
\end{tabular}
\end{center}

\vspace{0.7cm}

Since all the groups mentioned above are square-free and have exactly four divisors, Lemma \ref{F:lemma} ensures the existence of a unique modular form $F_N$ of weight 4 in $\Gamma_0(N)^\star$ satisfying $F_N(-1/N\tau) = -N^2 \tau^4 F_N(\tau)$, $F_N(i\infty) = 0$, and $L(F_N, 3) = \zeta(3).$ 

\begin{lemma}
    
These modular forms $F_N$, calculated following Lemma \ref{F:lemma}, are as follows:

\renewcommand{\arraystretch}{1.8}
\begin{center}
\begin{tabular}{|c|c|}
\hline
 $N: \Gamma_0(N)^\star$ & $F_N$ \\
 \hline
 $10$ & $ \frac7{432}E_4(\tau) - \frac{11}{36}E_4(2\tau) + \frac{275}{144}E_4(5\tau) - \frac{175}{108}E_4(10\tau) = $  \\
   & $\frac{35}{9}q - \frac{115}{3}q^2 + \frac{980}{9}q^3+O(q^4)$ \\
  \hline

 $14$ & $ \frac{1}{1560}\left(21E_4(\tau) - 364 E_4 (2\tau) + 4459 E_4( 7\tau) - 4116 E_4(14\tau) \right) = $ \\
 & $ \frac{42}{13}q - \frac{350}{13}q^2 + \frac{1176}{13}q^3 + O(q^4)$\\
  \hline

 $15$ & $ \frac1{112}\left( E_4(\tau) - 126E_4(3\tau) + 350E_4( 5\tau) -225 E_4(15\tau)\right) =  $ \\
 & $ \frac{15}{7}q + \frac{135}{7}q^2 - 210q^3 + O(q^4)$ \\

  \hline

 $21$ & $ \frac1{960}\left( 7E_4(\tau) - 693E_4(3\tau) + 3773E_4(7\tau) - 3087 E_4(21\tau) \right)  = $\\
 &  $ \frac{7}{4}q + \frac{63}{4}q^2 -\frac{497}{4}q^3 + O(q^4) $\\
  \hline

 $26$ & $  \frac{13}{1200}E_4(\tau) - \frac{39}{220}E_4(2\tau) + \frac{6591}{880}E_4(13\tau) - \frac{2197}{300}E_4(26\tau) =   $ \\
 & $ \frac{13}{5}q - \frac{1053}{55}q^2 + \frac{364}{5}q^3+O(q^4)$\\ 
  \hline

 $35$ & $ \frac1{1632}\left( 7E_4(\tau) -8925E_4(5\tau) + 17493 E_4(7\tau) -8575 E_4(35\tau) \right)  = $ \\
 & $ \frac{35}{34}q + \frac{315}{34}q^2 +\frac{490}{17}q^3 + O(q^4) $ \\
  \hline

 $39$ & $\frac1{4560}\left( 26E_4(\tau) - 2223 E_4(3\tau) + 41743E_4( 13\tau ) - 39546 E_4(39\tau) \right)    =  $ \\
 & $ \frac{26}{19}q + \frac{234}{19}q^2 - \frac{1495}{19}q^3 + O(q^4)$\\
\hline
\end{tabular}
\end{center}
\vspace{0.7cm}
\end{lemma}

Let $f_N$ be the Eichler integral associated with $F_N$, as defined in Lemma \ref{HeckeLemma}. Similar to the level $N=6$ case, there exists an affine subspace of functions $E_{N, \alpha}$ satisfying the conditions of Lemma \ref{E:lemma}. Consequently, the following identity holds:
\begin{align}
    E_{N, \alpha}(-1/N\tau)( f_N(-1/N\tau) - \zeta(3) ) = E_{N, \alpha}(\tau)(f_N(\tau) - \zeta(3)).
\end{align}

The explicit forms of these affine families are provided in the following lemma. As above, the expressions involving $E_2(d\tau)$ are to be read through Lemma \ref{lemma:E2_modular_combinations}; the listed Fricke relations force the quasimodular contribution to cancel.

\begin{lemma}
    For each $N \in \{10, 14, 15, 21, 26, 35, 39\}$, the affine family of weight two modular forms satisfying the required Fricke relation and normalization is given by $E_{N, \alpha} = E^1_N + \alpha E^0_N$, where
\begin{align*}
     E^1_{10}& = \frac53E_2(5\tau) - \frac23 E_2(2\tau) \\
    E^0_{10} &= E_2(\tau) -6E_2(2\tau) + 15E_2(5\tau) - 10E_2(10\tau)\\
    E_{14}^1 &= \frac75E_2(7\tau) - \frac25 E_2(2\tau) \\
    E_{14}^0 & = E_2(\tau) -\frac{26}{5}E_2(2\tau) + \frac{91}{5}E_2(7\tau) - 14E_2(14\tau) \\
    E_{15}^1 & = \frac52E_2(5\tau) - \frac32 E_2(3\tau)\\
    E_{15}^0 & = E_2(\tau) -21E_2(3\tau) + 35E_2(5\tau) - 15E_2(15\tau )\\ 
    E_{21}^1 & = \frac74 E_2(7\tau) - \frac34 E_2(3\tau) \\
    E_{21}^0 & = E_2(\tau) - 15 E_2(3\tau) + 35 E_2(7\tau) - 21E_2(21\tau) \\
    E_{26}^1 & = \frac{13}{11}E_2(13\tau) - \frac{2}{11}E_2(2\tau)\\
    E_{26}^0 & = E_2(\tau) -\frac{50}{11}E_2(2\tau) + \frac{325}{11}E_2(13\tau) - 26E_2(26\tau)\\
    E_{35}^1 & = \frac72 E_2(7\tau) - \frac52 E_2(5\tau) \\
    E_{35}^0  & = E_2(\tau) - 85E_2(5\tau) + 119E_2(7\tau) - 35E_2(35\tau) \\
    E_{39}^1 & = \frac{13}{10} E_2(13\tau) - \frac{3}{10}E_2(3\tau) \\
    E_{39}^0 & = E_2(\tau) - \frac{57}{5}E_2(3\tau) + \frac{741}{15} E_2(13\tau) - 39 E_2(39\tau).
\end{align*}
\end{lemma}

 We again consider the modular parametrization
\begin{align}
    \label{symmetry}
E_{N, \alpha}(\tau)( f_N(\tau) - \zeta(3)) = \sum_{n=0}^\infty (a_n^{N, \alpha} - b_n^{N, \alpha} \zeta(3))t_N(\tau)^n,
\end{align}
and compute a lower bound for the radius of convergence using Lemma \ref{lemma:branching}.

\begin{observation}\label{obs:branch_values}
    For $N = 6, 10, 14, 15, 21, 26, 35$ and $39$, the computations identify the following values of the Hauptmodul at the Fricke fixed point and at the first relevant branch point after it. The corresponding value of $|t_N(p_N)|$ gives the observed branch radius controlling the expansion of $E_{N, \alpha}(\tau)( f_N(\tau) - \zeta(3))$ in the coordinate $t_N$.
    

\vspace{0.3cm}
\renewcommand{\arraystretch}{1.6}
\begin{center}
\begin{tabular}{|c| c| c|}
\hline
 $N$: $\Gamma_0(N)^\star$ & $t_{N}(i/\sqrt{N})$ & $t_{N}(p_N)$\\
 \hline
 $6$ & $(\sqrt{2}-1)^4$ &  $(\sqrt{2}+1)^4$\\
 \hline
  $10$ & $\approx 0.0557$ & $\approx 1$\\
 \hline
 $14$  & $\approx 0.0795$ & $\approx 1$\\
 \hline
  $15$  & $\approx 0.0901$ & $\approx 1.6180$\\
\hline
$21$ & $\approx 0.1224$ & $\approx 0.5865$\\
 \hline
  $26$  & $\approx 0.1385$ & $\approx 0.8134$\\
 \hline
 $35$  & $\approx 0.1878$ & $\approx 0.6180$\\
\hline
 $39$ &  $\approx 0.1958$ & $\approx 0.5806$\\
\hline
\end{tabular}
\end{center}
\vspace{0.7cm}

\end{observation}

\subsection{Geometric obstruction to the irrationality argument}

Assuming the branch values in Observation \ref{obs:branch_values} are the exact first relevant branch values, the table separates the level $6$ case from the other genus-zero examples. The standard Beukers argument does not only require that the linear forms
\begin{align*}
    a_n^{N,\alpha}-b_n^{N,\alpha}\zeta(3)
\end{align*}
decay; it requires that they still tend to zero after clearing denominators. In the constructions considered here, the denominator contribution coming from the Eichler integral is controlled by $\operatorname{lcm}(1,\ldots,n)^3$, hence by $e^{(3+\epsilon)n}$. On the other hand, the geometric decay is governed by the branch radius $R_N=|t_N(p_N)|$. Thus the usual linear-form argument can only succeed when the geometric inequality
\begin{align*}
    R_N>e^3
\end{align*}
is available, up to the harmless $\epsilon$-losses in the estimates.

For $N=6$, one has $R_6=(\sqrt2+1)^4\approx 33.97$, while $e^3\approx 20.09$. This is the numerical gap that makes the Apéry-Beukers irrationality proof possible. For the other levels in the table, the observed values satisfy $R_N\le 1.6181$, and hence are far below $e^3$. Consequently, although the same modular mechanism produces rational linear forms, this direct Beukers-type denominator estimate cannot yield an irrationality proof for these levels.

\begin{remark}
Some of the branch values appear to be simple algebraic numbers. For example, the numerical values for $N=15$ and $N=35$ are consistent with the golden ratio $\frac{1+\sqrt5}{2}$ and its inverse, respectively. The levels $21$, $26$, $35$ and $39$ instead have branch values below $1$, which means that the corresponding coefficient bounds do not even give exponential decay of the linear forms in this coordinate. This reinforces the geometric nature of the obstruction: the failure is not caused by the choice of parameter $\alpha$, but by the position of the next branch point of the Hauptmodul.
\end{remark}

For none of the levels $N\neq 6$ considered here does the observed radius give the exponential decay needed for a Beukers-type irrationality argument. Moreover, when the relevant radius is at most $1$, the estimates do not even imply convergence of the linear forms $|a_n^{N,\alpha}-b_n^{N,\alpha}\zeta(3)|$ to zero. Nevertheless, the resulting rational approximants are still natural objects to compute. We therefore record the first few values of $a_n^{N,\alpha}/b_n^{N,\alpha}$ for these levels and for several choices of the parameter $\alpha$, as numerical evidence for how the affine parameter changes the finite-level approximations.

We also carried out the same computations for additional levels outside the main list discussed above, including $N=8,12,18,20,$ and $50$. These auxiliary examples are included in the table to illustrate that the same coefficient-extraction procedure can be applied more broadly, even when the resulting approximations do not support an irrationality argument. The classical Apéry row and the level $N=6$ rows are included as benchmarks for the Apéry-Beukers case, where the denominator-adjusted estimates are strong enough for the irrationality argument.

\vspace{0.3cm}
\begingroup
\scriptsize
\renewcommand{\arraystretch}{1.35}
\begin{center}
\begin{tabular}{|c|c|c|c|c|}
\hline
$N$ & $\alpha$ & $a_5^{N,\alpha}/b_5^{N,\alpha}$ & $\log_{10}\operatorname{den}(a_{199}^{N,\alpha}/b_{199}^{N,\alpha})$ & error at $n=199$ \\
\hline
$\mathrm{Ap\acute ery}$ & classical & $1.2020569032$ & $567.4$ & $1.52\cdot 10^{-609}$ \\
\hline
$6$ & $0$ & $1.2020569032$ & $561.0$ & $3.04\cdot 10^{-607}$ \\
$6$ & $1$ & $1.2020569032$ & $564.3$ & $5.21\cdot 10^{-608}$ \\
\hline
$8$ & $0$ & $1.2020574348$ & $520.9$ & $3.65\cdot 10^{-304}$ \\
$8$ & $1$ & $1.2020569634$ & $524.2$ & $3.03\cdot 10^{-305}$ \\
\hline
$10$ & $0$ & $1.2020522755$ & $508.9$ & $1.70\cdot 10^{-249}$ \\
$10$ & $1$ & $1.2020516590$ & $513.1$ & $1.94\cdot 10^{-249}$ \\
\hline
$14$ & $0$ & $1.2020634126$ & $481.4$ & $1.15\cdot 10^{-218}$ \\
$14$ & $1$ & $1.2020644201$ & $483.0$ & $1.26\cdot 10^{-218}$ \\
\hline
$12$ & $0$ & $1.2050000000$ & $343.7$ & $9.37\cdot 10^{-224}$ \\
$12$ & $1$ & $1.2061302049$ & $342.5$ & $1.27\cdot 10^{-223}$ \\
\hline
$15$ & $0$ & $1.2021649910$ & $468.3$ & $3.54\cdot 10^{-248}$ \\
$15$ & $1$ & $1.2021753512$ & $470.7$ & $3.83\cdot 10^{-248}$ \\
\hline
$18$ & $0$ & $1.1960565476$ & $310.8$ & $7.09\cdot 10^{-195}$ \\
$18$ & $1$ & $1.1944794438$ & $314.7$ & $8.64\cdot 10^{-195}$ \\
\hline
$20$ & $0$ & $1.1875000000$ & $299.4$ & $1.69\cdot 10^{-179}$ \\
$20$ & $1$ & $1.1827153110$ & $299.9$ & $1.95\cdot 10^{-179}$ \\
\hline
$21$ & $0$ & $1.2011367917$ & $438.6$ & $2.25\cdot 10^{-181}$ \\
$21$ & $1$ & $1.2010710819$ & $445.9$ & $2.40\cdot 10^{-181}$ \\
\hline
$35$ & $0$ & $1.2878789488$ & $400.9$ & $9.90\cdot 10^{-134}$ \\
$35$ & $1$ & $1.2936280286$ & $408.2$ & $1.04\cdot 10^{-133}$ \\
\hline
$39$ & $0$ & $1.2181947829$ & $400.8$ & $3.31\cdot 10^{-142}$ \\
$39$ & $1$ & $1.2187952773$ & $405.4$ & $4.12\cdot 10^{-142}$ \\
\hline
$50$ & $0$ & $1.9461050725$ & $246.6$ & $4.59\cdot 10^{-106}$ \\
$50$ & $1$ & $2.0980991224$ & $251.3$ & $5.07\cdot 10^{-106}$ \\
\hline
\end{tabular}
\end{center}
\endgroup
\vspace{0.3cm}

The entries in the third column are shown as decimal approximations to keep the table readable. We omit the decimal expansion of $a_{199}^{N,\alpha}/b_{199}^{N,\alpha}$ itself, since in the rapidly convergent cases it agrees with $\zeta(3)$ to all displayed digits. The fourth column gives the base-ten logarithm of the denominator of the reduced rational number $a_{199}^{N,\alpha}/b_{199}^{N,\alpha}$, and the last column records the absolute error $\left|a_n^{N,\alpha}/b_n^{N,\alpha}-\zeta(3)\right|$ at $n=199$.

It is useful to compare the error with the size of the denominator. In the following table we write
\begin{align*}
    D_{199}^{N,\alpha}&=\log_{10}\operatorname{den}\left(a_{199}^{N,\alpha}/b_{199}^{N,\alpha}\right),\\
    E_{199}^{N,\alpha}&=-\log_{10}\left|a_{199}^{N,\alpha}/b_{199}^{N,\alpha}-\zeta(3)\right|,\\
    Q_{199}^{N,\alpha}&=E_{199}^{N,\alpha}/D_{199}^{N,\alpha}.
\end{align*}
The quantity $Q_{199}^{N,\alpha}$ measures the quality of the approximation relative to the size of the denominator. 

\vspace{0.3cm}
\begingroup
\scriptsize
\renewcommand{\arraystretch}{1.35}
\begin{center}
\begin{tabular}{|c|c|c|c|}
\hline
$N$ & $\alpha$ & $E_{199}^{N,\alpha}$ & $Q_{199}^{N,\alpha}$ \\
\hline
$\mathrm{Ap\acute ery}$ & classical & $608.8$ & $1.073$ \\
\hline
$6$ & $0$ & $606.5$ & $1.081$ \\
$6$ & $1$ & $607.3$ & $1.076$ \\
\hline
$8$ & $0$ & $303.4$ & $0.583$ \\
$8$ & $1$ & $304.5$ & $0.581$ \\
\hline
$10$ & $0$ & $248.8$ & $0.489$ \\
$10$ & $1$ & $248.7$ & $0.485$ \\
\hline
$14$ & $0$ & $217.9$ & $0.453$ \\
$14$ & $1$ & $217.9$ & $0.451$ \\
\hline
$12$ & $0$ & $223.0$ & $0.649$ \\
$12$ & $1$ & $222.9$ & $0.651$ \\
\hline
$15$ & $0$ & $247.5$ & $0.528$ \\
$15$ & $1$ & $247.4$ & $0.526$ \\
\hline
$18$ & $0$ & $194.1$ & $0.625$ \\
$18$ & $1$ & $194.1$ & $0.617$ \\
\hline
$20$ & $0$ & $178.8$ & $0.597$ \\
$20$ & $1$ & $178.7$ & $0.596$ \\
\hline
$21$ & $0$ & $180.6$ & $0.412$ \\
$21$ & $1$ & $180.6$ & $0.405$ \\
\hline
$35$ & $0$ & $133.0$ & $0.332$ \\
$35$ & $1$ & $133.0$ & $0.326$ \\
\hline
$39$ & $0$ & $141.5$ & $0.353$ \\
$39$ & $1$ & $141.4$ & $0.349$ \\
\hline
$50$ & $0$ & $105.3$ & $0.427$ \\
$50$ & $1$ & $105.3$ & $0.419$ \\
\hline
\end{tabular}
\end{center}
\endgroup
\vspace{0.3cm}

\subsection{Computational reproducibility}

We briefly summarize the computations used in this section. For each square-free level $N$ with four divisors, the form $F_N$ was obtained by solving the linear system in Lemma \ref{F:lemma} in the Eisenstein basis $\{E_4(d\tau):d\mid N\}$. The affine family $E_{N,\alpha}$ was computed in the auxiliary coefficient space spanned by the functions $E_2(d\tau)$, using the Fricke symmetry and the normalization at $i\infty$; Lemma \ref{lemma:E2_modular_combinations} then identifies the resulting combinations as genuine weight two modular forms. The listed eta-products for the Hauptmoduln were then expanded as $q$-series and inverted formally to express
\begin{align*}
    E_{N,\alpha}(\tau)(f_N(\tau)-\zeta(3))
\end{align*}
as a power series in $t_N(\tau)$.

The branch values were computed by evaluating the eta-product expressions for $t_N$ at the corresponding branch points to high precision using PARI/GP.

\end{document}